\documentclass[12pt]{article}
\usepackage[T1]{fontenc}
\usepackage[utf8]{inputenc}
\usepackage{lmodern}
\usepackage{graphicx} 
\usepackage{amsmath}
\usepackage{amssymb}
\usepackage{amsthm}
\usepackage{mathtools}
\usepackage{xcolor}

\usepackage{tikz-cd}
\usepackage{xurl}
\usepackage{hyperref}

\usepackage{enumitem}

\usepackage[backend=biber,style=numeric-comp,sorting=none]{biblatex}
\newtheorem{thm}{Theorem}[section]
\newtheorem{prop}[thm]{Proposition}

\newtheorem{cor}[thm]{Corollary}

\theoremstyle{definition}
\newtheorem{rem}[thm]{Remark}

\newcommand{\bN}{\mathbb{N}}
\newcommand{\bC}{\mathbb{C}}
\newcommand{\bT}{\mathbb{T}}
\newcommand{\bZ}{\mathbb{Z}}
\newcommand{\bR}{\mathbb{R}}
\newcommand{\bD}{\mathbb{D}}

\newcommand{\cS}{\mathcal{S}}
\newcommand{\cA}{\mathcal{A}}
\newcommand{\cF}{\mathcal{F}}
\newcommand{\cC}{\mathcal{C}}
\newcommand{\cD}{\mathcal{D}}
\newcommand{\cB}{\mathcal{B}}

\title{Commutative topological algebras on translation-invariant reproducing kernel Hilbert spaces}
\author{Miguel Angel Rodriguez Rodriguez}

\begin{document}

\maketitle

\begin{abstract}
We study commutative topological algebras naturally associated with translation-invariant reproducing kernel Hilbert spaces whose direct integral decomposition has one-dimensional fibers. 
Starting from the bounded algebra of translation-invariant operators, we pass to a common dense domain generated by reproducing kernels and identify the corresponding diagonalizable operators with multiplication by symbols in an intersection of weighted $L^2$-spaces.  On the symbol side this gives a canonical space $\mathcal F_0$ and a maximal multiplicative subalgebra $\mathcal F_M$, which is a complete locally convex $*$-algebra.  Transporting the structure back yields corresponding algebras of operators and integral kernels.  We also discuss when the inclusions
\[
L^\infty(\Omega)=\mathcal F_\infty\subset \mathcal F_M\subset \mathcal F_0
\]
are strict, and illustrate the results with vertical and radial operators on classical Bergman and Fock spaces.
\end{abstract}



\section{Introduction}

Reproducing kernel Hilbert spaces, or RKHSs, and the operators acting on them form an active area of research with connections to function theory, operator theory, harmonic analysis, quantum mechanics and machine learning.

In this paper we study RKHSs $H$ embedded in $L^2(G\times Y,\nu\otimes\lambda)$, where $G$ is a locally compact abelian group with Haar measure $\nu$, $(Y,\lambda)$ is a measure space, and $H$ is invariant under translations in the $G$-variable.  The precise assumptions are recalled in Section \ref{sec:preliminaries}.  This setting includes, in some cases after a unitary transformation, classical examples such as Bergman, Fock and Hardy spaces.

In \cite{HerreraYanezMaximenkoRamosVazquez2022}, it was shown that spaces of this type admit a direct integral representation
\[
R\colon H\longrightarrow \int_{\Omega} H_\xi\,d\widehat{\nu}(\xi),
\]
where $\Omega\subset \widehat{G}$ and $(H_\xi)_{\xi\in\Omega}$ is a measurable family of Hilbert spaces.  When the fibers $H_\xi$ are one-dimensional, this representation identifies the algebra $\cS_\infty$ of bounded $G$-translation-invariant operators on $H$ with the multiplication algebra $L^\infty(\Omega)$.  Thus $\cS_\infty$ is a commutative $W^*$-algebra.

A second realization of the same algebra is obtained from integral kernels.  Indeed, every bounded operator $T$ on $H$ has an integral kernel $K_T$ defined in terms of the reproducing kernel of $H$.  In \cite{BaisMaximenkoVenkuNaidu2025}, the algebraic structure of $L^\infty(\Omega)$ was transported to these kernels.  This gives a commutative $W^*$-algebra $\cA_\infty$ of functions on $G\times Y\times Y$, with a convolution-type product denoted here by $\odot$.  In the bounded case, one therefore has three equivalent models:
\[
\cS_\infty\longleftrightarrow \cA_\infty\longleftrightarrow L^\infty(\Omega).
\]

The purpose of this paper is to extend this picture to a natural class of unbounded operators.  Let
\[
\cD_0=\operatorname{span}\{K_{u,v}\colon (u,v)\in G\times Y\}
\]
be the dense subspace generated by the reproducing kernels of $H$.  The bounded kernel algebra $\cA_\infty$ is contained in a larger space $\cA_0$, whose elements define possibly unbounded diagonalizable operators on the common domain $\cD_0$.  We denote the corresponding space of operators by $\cS_0$.
The question motivating this work is how far the product $\odot$ on $\cA_\infty$ can be extended inside $\cA_0$, and what this extension means on the operator and symbol sides.

Our first main result identifies the symbol model for $\cS_0$.  We show that the operators in $\cS_0$ are represented, through $R$, by multiplication operators whose symbols belong to a space
\[
\cF_0=\bigcap_{y\in Y} L^2(\Omega,\mu_y).
\]
We then introduce
\[
\cF_M=\{a\in\cF_0\colon a^m\in\cF_0\text{ for every }m\in\bN\}.
\]
This is the largest subalgebra of $\cF_0$ with respect to pointwise multiplication.  We prove that $\cF_M$ is a complete commutative topological $*$-algebra and that
\[
L^\infty(\Omega)=\cF_\infty\subset\cF_M\subset\cF_0.
\]

Transporting this structure back to operators and kernels gives intermediate algebras
\[
\cS_\infty\subset\cS_M\subset\cS_0,
\qquad
\cA_\infty\subset\cA_M\subset\cA_0.
\]
The algebra $\cA_M$ is the largest subalgebra of $\cA_0$ to which the kernel product can be extended through the symbol correspondence.  More precisely, the product on $\cA_M$ is transported from pointwise multiplication on $\cF_M$ and agrees with the formal integral product whenever that integral product defines an element of $\cA_0$.  Hence the problem of extending the bounded kernel product to all of $\cA_0$ is reduced to the problem of understanding when $\cF_M=\cF_0$.  We do not give a complete characterization of this equality, but we prove several partial results and present examples showing that the inclusions above may be strict.

The paper is organized as follows.  In Section \ref{sec:translation invariant RKHS} we recall the direct integral representation of translation-invariant RKHSs.  In Section \ref{sec:translation-invariant operators} we review the bounded theory and introduce the spaces $\cS_0$, $\cA_0$ and $\cF_0$.  Section \ref{sec:topological algebras} contains the construction of $\cF_M$ and the proof of its main algebraic and topological properties.  In Section \ref{sec:algebras A S} we transfer these results to the operator and kernel realizations.  Section \ref{sec:examples} discusses examples coming from the upper half-plane, the Bergman space on the unit disk and the Fock space.  Finally, Section \ref{sec:questions} collects some open questions.

Most existing work in this direction concerns bounded operator algebras, especially Banach algebras, $C^*$-algebras and von Neumann algebras.  There is a substantial literature on unbounded operators on Bergman, Fock and Hardy spaces; just to mention a few works, see for example, \cite{HanWangWu2022,Haslinger2024,ArreortuaLoaizaMacias2026,Sarason2008Toeplitz}.  However, little seems to have been said about commutative algebras generated by unbounded operators on RKHSs such as Bergman and Fock spaces.  The novelty of the present work is to show that, in the translation invariant setting, these algebras admit an explicit symbol model and a largest natural commutative extension of the bounded theory.

\section{Translation-invariant RKHS}
\label{sec:translation invariant RKHS}

\subsection{Preliminaries}
\label{sec:preliminaries}

We recall some facts from the previous works \cite{HerreraYanezMaximenkoRamosVazquez2022,BaisMaximenkoVenkuNaidu2025}.
Let $G$ be a locally compact abelian group with Haar measure $\nu$ and $Y$ a measure space with measure $\lambda$. We assume that $G$ is metrizable and $\sigma$-compact, that $Y$ is $\sigma$-finite and that the spaces $L^2(G,d\nu)$ and $L^2(Y,d\lambda)$ are separable.

We consider a Hilbert space $H$ of functions on $G\times Y$ embedded in the space $L^2(G\times Y)$ and invariant under translations of $G$.
We will denote the unitary representation of $G$ on $H$ by
\[
(\rho(p)f)(x,y)=f(x-p,y),\qquad x\in G,\ y\in Y,\ p\in G.
\]

Throughout the paper, inner products on complex Hilbert spaces are taken to be linear in the first variable and conjugate-linear in the second one. Thus the reproducing property is written as
\[
f(x,y)=\langle f,K_{x,y}\rangle_H,\qquad x\in G,\ y\in Y.
\]

Let $(K_{x,y})_{(x,y)\in G\times Y}$ denote the reproducing kernel of $H$. One can check that the $G$-invariance of $H$ is equivalent to the condition
\begin{equation}
  K_{x,y}(u,v)=K_{0,y}(u-x,v),\qquad \forall \ (x,y),(u,v)\in G\times Y.  
\end{equation}
In order for some integrals to be well-defined, we will further suppose that
\begin{equation}
\label{eq:sup condition Kxy}
\sup_{v\in Y}\int_G |K_{0,y}(u,v)|d\nu (u) < +\infty,\qquad \forall y\in Y.
\end{equation}

\subsection{Fourier transforms and direct integral representation}

Let $(\widehat{G},d\widehat{\nu})$ be the dual group of $G$, where  $\widehat{\nu}$ is its Haar measure, normalized so that the Fourier transform from $L^2(G,\nu)$ to $L^2(\widehat{G},d\widehat{\nu})$ is unitary.

We denote by $L_{\bullet,y}(v)$ the Fourier transform of $K_{0,y}(\cdot,v)$, where $y,v\in Y$ are fixed and $\bullet$ serves as a placeholder for the variable in $\widehat{G}$. That is:
\[
L_{\xi,y}(v)=\int_G K_{0,y}(u,v)\overline{\xi(u)}d\nu(u),\qquad \xi\in\widehat{G},\ y,v\in Y.
\]
Let $\Omega$ be the space of ``non-trivial frequencies'' in the dual space $\widehat{G}$:
\[
\Omega = \{\xi \in \widehat{G} \colon\  \exists y\in Y\ L_{\xi,y}(y)\neq 0\}.
\]
We will further assume that there is a measurable family of functions $(q_\xi)_{\xi\in\Omega}$ such that
\begin{equation}
\label{eq:assumption L=qq} 
L_{\xi,y}(v)=\overline{q_\xi(y)}q_\xi(v),\qquad \xi\in\Omega,\ v,y\in Y.
\end{equation}
Moreover, for every $\xi\in\Omega$ we denote by $H_\xi$ the one-dimensional space $\bC q_\xi$.

\medskip
\noindent\textbf{Standing assumptions.}
From this point on, unless explicitly stated otherwise, we assume that $G$ is a metrizable $\sigma$-compact locally compact abelian group, $(Y,\lambda)$ is $\sigma$-finite, the spaces $L^2(G,d\nu)$ and $L^2(Y,d\lambda)$ are separable, $H\subset L^2(G\times Y,\nu\otimes\lambda)$ is a translation-invariant RKHS satisfying \eqref{eq:sup condition Kxy}, and the one-dimensional fiber condition \eqref{eq:assumption L=qq} holds.
\medskip

The following easy consequence will be useful later on. As before, given $y\in Y$, we denote by $q_\bullet(y)$ the function $\xi\mapsto q_\xi(y)$.
\begin{cor}
\label{cor:|q| continuous in xi}
    For every fixed $y\in Y$, the function $|q_\bullet(y)|$ is continuous.
\end{cor}
\begin{proof}
    As is well-known, the Fourier transform maps $L^1$ functions into $C_0(\widehat{G})$, which implies that $L_{\xi,y}(y)$ is continuous.
    By the assumption \eqref{eq:assumption L=qq}, we have $L_{\xi,y}(y)=|q_\xi(y)|^2$, whence the conclusion.
\end{proof}

Consider now the operator $R=N\mathbf{F}$, where
\[
\mathbf F=F\otimes I\colon L^2(G\times Y)\to L^2(\widehat{G}\times Y)
\]
is the Fourier-Plancherel transform on the first component and
\[
N\colon L^2(\widehat{G}\times Y)\to 
L^2(\Omega,\widehat{\nu})
=
\int_{\Omega}^\oplus
H_\xi d\widehat{\nu}(\xi)
\]
acts by
\[
(Ng)(\xi)=\langle g(\xi,\cdot),q_\xi\rangle_{L^2(Y)},\qquad\xi\in \Omega.
\]
Herrera-Yanez, Maximenko and Ramos-Vazquez showed that $R|_H$ is a unitary map. For simplicity, we will simply write $R$ instead of $R|_H$. Note that the direct integral decomposition of $H$ is one-dimensional.

The inverse of $R$ is the operator $R^*=\mathbf{F}^*N^*$ given by
\begin{equation}
\label{eq:def Rstar}
R^* f(x,v)
=
\int_{\Omega}f(\xi) q_\xi(v) \xi(x)  d\widehat{\nu}(\xi),\qquad x\in G,\ v\in Y. 
\end{equation}

Under this map, translations become unitary multiplication operators:
\begin{equation}
    \label{eq:Rrho x R* equals Ex}
    R\rho(x)R^*=M_{E_{-x}},
\end{equation}
where $E_x$ is the function given by
\[
E_x(\xi)=\xi(x),\qquad \xi\in \Omega.
\]

We recall the following useful fact.
\begin{prop}[{\cite[Proposition 3.3]{BaisMaximenkoVenkuNaidu2025}}]
    \label{prop:RK=q}
    Let $x\in G$ and $y\in Y$. Then
    \[
    RK_{x,y}=E_{-x}\overline{q_\bullet(y)}.
    \]
\end{prop}
We also recall the following useful partition of $\Omega$:
\begin{prop}[{\cite[Proposition 3.4]{BaisMaximenkoVenkuNaidu2025}}]
\label{prop:partition Y0}
    There exists a finite or countable subset $Y_0$ of $Y$ and a measurable partition $(\Omega_y)_{y\in Y_0}$ such that for every $\xi$ in $\Omega_y$, $L_{\xi,y}(y)>0$.
\end{prop}

Throughout the paper we will make use of the following notation.
For $y\in Y$, we consider the open set
\begin{equation}
    \label{eq:definition Uy}
    U_y\coloneqq
\big\{\xi\in\Omega\colon q_\xi(y)\neq 0\big\}
=
\big\{\xi\in\Omega\colon L_{\xi,y}(y)>0\big\},
\end{equation}
and define $\mu_y$ as the measure given by
\begin{equation}
    \label{eq:definition muy}
\mu_y(E)=\int_{E}|q_\xi(y)|^2d\widehat{\nu}(\xi),\qquad E\subset \Omega\text{ measurable}.
\end{equation}
Note that $\mu_y$ is a finite measure concentrated in $U_y$, since
\begin{align*}
\mu_y(\Omega)&=\int_\Omega |q_\xi(y)|^2d\widehat{\nu}(\xi)
=
\int_{U_y}|q_\xi(y)|^2d\widehat{\nu}(\xi)\\
&=
\|q_\bullet(y)\|^2_{L^2(\Omega,\widehat{\nu})}
=
\|K_{0,y}\|_H^2=K_{0,y}(0,y)<\infty.    
\end{align*}

\subsection{\texorpdfstring{On unbounded operators defined on $H$}{On unbounded operators defined on H}}

Denote by $\cB(H)$ the space of bounded operators on $H$. The commutant $\cC(\rho)$ of the representation $\rho$ is defined by
\[
\cC(\rho)=\{T\in \cB(H)\colon \rho(x)T=T\rho(x),\ \forall x\in G\}.
\]
Such operators will be called \emph{translation invariant}.
In this work we will systematically deal with unbounded operators. As usual, we will say ``unbounded'' to refer to possibly unbounded.

We will mainly consider unbounded operators defined on the domain
\begin{equation}
    \label{eq:def domain D0}
    \cD_0=\operatorname{span}\{K_{x,y}\colon x\in G,\ y\in Y\}.
\end{equation}
It is clearly dense in $H$ and invariant under the translations $\rho(x)$, $x\in G$.
Note that, by Proposition \ref{prop:RK=q}, we have
\begin{equation}
    \label{eq:def domain RD0}
    R(\cD_0)=\operatorname{span}\{E_{-x}\overline{q_\bullet(y)}\colon x\in G,\ y\in Y\}\subset L^2(\Omega,\widehat{\nu}).
\end{equation}

Recall that an operator $T\colon \operatorname{Dom}(T)\to H$ is called closed if its graph
\[
\operatorname{Graph}(T)=\{(f,Tf)\colon f\in\operatorname{Dom}(T)\}
\]
is \emph{closed} in $H\oplus H$ and \emph{closable} if the closure of its graph corresponds to the graph of an operator.
The unitarity of $R$ implies that an operator $T\colon \operatorname{Dom}(T)\to H$ is closable if and only if $RTR^*\colon R(\operatorname{Dom}(T))\to R(H)$ is closable replacing $H$ by $R(H)=L^2(\Omega,\widehat{\nu})$ in the above definition.

Moreover, we will systematically make use of multiplication operators.
If \(b\colon \Omega\to\mathbb C\) is measurable and finite almost everywhere, 
we write $M_b$ for the multiplication operator on $R(\cD_0)$:
\[
M_bf=bf,\qquad f\in R(\cD_0).
\]
Note that $M_b\subset M_b^{\text{max}}$, where $M_b^{\text{max}}$ denotes the
maximal multiplication operator on 
\(L^2(\Omega,\widehat\nu)\), with domain
\[
\operatorname{Dom}(M_b^{\text{max}})=\{f\in L^2(\Omega,\widehat\nu): bf\in L^2(\Omega,\widehat\nu)\}.
\]

\section{Operator, kernel and symbol correspondence}
\label{sec:translation-invariant operators}

\subsection{Translation-invariant operators and their kernels}

Every bounded operator $T$ on $H$ can be written as an integral operator. In \cite[Proposition 2.2]{BaisMaximenkoVenkuNaidu2025}, the authors showed specifically that
\[
(Tf)(x,y)=\int_{G\times Y}K_T(x,y,u,v)f(u,v) d\nu(u)d\lambda(v),
\]
where the function $K_T$ is defined by
\[
K_T(x,y,u,v)=(TK_{u,v})(x,y).
\]
Moreover, they proved that if $T$ is translation invariant, then
\[
K_T(x,y,u,v)=K_T(x-u,y,0,v)
\]
and that the function 
\[
G\times Y\times Y\ni (x,y,v)\mapsto K_T(x,y,0,v)\in \bC,
\]
belongs to the class $\cA_0$ of functions
$\psi\colon G\times Y\times Y\to \bC$ such that
\begin{equation}
\label{eq:A0 condition 1}
    \psi(\cdot,\cdot,v)\in H,\qquad v\in Y,
\end{equation}
\begin{equation}
\label{eq:A0 condition 2}
    \overline{\psi(-\cdot,y,\cdot)}\in H,\qquad y\in Y.
\end{equation}
This space is endowed with an involution
\begin{equation}
    \label{eq:def involution A0}
    \phantom{\psi}^\dagger\colon \cA_0\longrightarrow \cA_0
\end{equation}
given by
\[
\psi^\dagger(x,y,v)=\overline{\psi(-x,v,y)},\qquad x\in G,\ y,v\in Y.
\]

Given $\psi\in\mathcal A_0$, we define the formal integral operator $\widetilde{S}_\psi$ by
\begin{equation}
\label{eq:def S tilde psi}  
\widetilde{S}_\psi f(x,y)
=
\int_{G\times Y} f(u,v)\psi(x-u,y,v) d\nu (u)d\lambda(v),
\end{equation}
Note that, by \eqref{eq:A0 condition 2}, we have $\psi(\cdot,y,\cdot)\in L^2(G\times Y)$, for every $y\in Y$. Hence, by the invariance of $\nu$ and Hölder's inequality, we have
\[
\int_{G\times Y}|f(u,v)||\psi(x-u,y,v)| d\nu(u)d\lambda(v)
<\infty.
\]
So in fact the integral \eqref{eq:def S tilde psi} exists for every $f\in H$, though the result may of course fail to be in $H$.

As was shown in \cite[Lemma 5.3]{BaisMaximenkoVenkuNaidu2025}, such an operator $\widetilde{S}_\psi$ is well-defined on the kernel functions $K_{p,q}$, $(p,q)\in G\times Y$, and one has
\begin{equation}
    \label{eq:tildeSK is psi}
    (\widetilde{S}_\psi K_{p,q})(x,y)
=
\psi(x-p,y,q),\qquad x\in G,\ y\in Y,
\end{equation}
which shows that the application $\psi\mapsto \widetilde{S}_\psi$ is injective.

Since $\widetilde S_\psi$ is linear and is defined pointwise on all of $H$, its restriction to the span of the kernel functions is well-defined. Thus we define
\[
S_\psi=\widetilde S_\psi|_{\cD_0}.
\]
Furthermore, we define the space
\begin{equation}
    \label{eq:definition S0}
    \cS_0=\{S_\psi\colon \psi\in \cA_0\}.
\end{equation}
One can easily check that for $\alpha,\beta\in\bC$ and $\psi_1,\psi_2\in\cA_0$ it holds
\[
S_{\alpha\psi_1+\beta\psi_2}=\alpha S_{\psi_1}+\beta S_{\psi_2}.
\]
Furthermore, denote by $\mathcal A_\infty$ the set of all $\psi$ in $\mathcal A_0$ such that the operator $S_\psi$ has a bounded extension. 
When $S_\psi$ has a bounded extension to $H$, we identify $S_\psi$ with that unique bounded extension.
We remark that this is a different convention from the one used in \cite{BaisMaximenkoVenkuNaidu2025}, but we choose this one for simplicity. In analogy to $\cA_\infty$, we set
\[
\cS_\infty=\{S_\psi\colon \psi\in \cA_\infty\}.
\]
Moreover, we set $\|\psi\|_{\cA_\infty}\coloneqq \|S_\psi\|_{op}$.

Furthermore, there is a convolution-type product that one can define on the elements of $\cA_0$.
For $\varphi,\psi\in\cA_0$ we define the product
\begin{equation}
\label{eq:def odot product}
(\varphi\odot\psi)(x,y,v)=
\int_{G\times Y}
\varphi(x-s,y,t)\psi(s,t,v)d\nu(s)d\lambda(t).
\end{equation}
For each fixed $(x,y,v)$, the integral is finite by the Cauchy--Schwarz inequality, because the two factors are functions of $(s,t)\in G\times Y$ belonging to $L^2(G\times Y)$.  The resulting function need not belong to $\cA_0$, so at this stage \eqref{eq:def odot product} should be regarded as a formal product on $\cA_0$.

As it turns out, the operator composition in $\cS_\infty$ translates as the product $\odot$ in $\cA_\infty$. Indeed, as was shown in \cite[Lemma 7.6]{BaisMaximenkoVenkuNaidu2025}, if $\varphi,\psi\in\cA_\infty$ one has
\[
S_{\varphi}S_{\psi}=S_{\varphi\odot\psi}.
\]

\begin{prop}
    \label{prop:tilde S adjoint closable}
    Let $\psi\in\cA_0$. Then
    $S_{\psi^\dagger}$ is a formal adjoint for $S_\psi$ in the sense that
    \[
    \langle S_{\psi}f,g\rangle=
    \langle f,S_{\psi^\dagger}g\rangle,
    \]
    for all $f,g\in\cD_0$.
    In particular, \(S_\psi\) is closable and $S_{\psi^\dagger}\subset S_\psi^*$.
\end{prop}
\begin{proof}
Let $(p,q),(x,y)\in G\times Y$. By \eqref{eq:tildeSK is psi} and the reproducing property we have
\[
\langle
S_\psi K_{p,q},K_{x,y}
\rangle
=S_\psi K_{p,q}(x,y)
=\psi(x-p,y,q).
\]
On the other hand, by
\eqref{eq:A0 condition 2}, \eqref{eq:def involution A0} and \eqref{eq:def S tilde psi}, we have
\[
    \psi(x-p,y,q)
    =
    \overline{
    \psi^\dagger(p-x,q,y)
    }
    =
    \overline{S_{\psi^\dagger} K_{x,y}(p,q)}
    =
    \overline{
    \langle
    S_{\psi^\dagger} K_{x,y},K_{p,q}
    \rangle
    }.
\]
Therefore,
\[
\langle
S_\psi K_{p,q},K_{x,y}
\rangle=
    \langle
    K_{p,q},S_{\psi^\dagger} K_{x,y}
    \rangle,
\]
and the proposition follows.
\end{proof}

\subsection{Diagonalization of unbounded translation-invariant operators}

In \cite{BaisMaximenkoVenkuNaidu2025} the authors proved that $\cS_\infty=\cC(\rho)$ and that every operator in $\cS_\infty$ is isomorphic via $R$ to a bounded multiplication operator acting on $L^2(\Omega)$. 

We now show that this diagonalization extends to the closable translation-invariant unbounded operators. The argument is rather standard, but we include a proof for completeness.

If $S:\cD_0\to H$ is an operator defined on the common domain $\cD_0$, we say that $S$ is \emph{translation invariant} if
\[
S\rho(x)f=\rho(x)Sf,\qquad f\in\cD_0,\ x\in G.
\]
Since $\cD_0$ is invariant under the translations, both sides are well-defined.
It follows easily from \eqref{eq:tildeSK is psi} that every element of $\cS_0$ is translation invariant.

For a subspace $X\subset L^\infty(\Omega)$, denote by $\mathcal M(X)$ the space
of bounded multiplication operators with symbols in $X$:
\[
\mathcal M(X)=\{M_a\colon a\in X\}.
\]
In particular, $\mathcal M(L^\infty(\Omega))$ is the multiplication algebra on
$L^2(\Omega,\widehat{\nu})$.
We are particularly interested in the case where $X$ is the space
\[
\mathcal E=\operatorname{span}\{E_x\colon x\in G\}\subset L^\infty(\Omega).
\]
\begin{prop}
\label{prop:WOT SOT closure ME is MLinfty}
It holds
    \[
\overline{\mathcal M(\mathcal E)}^{\mathrm{WOT}}
=
\overline{\mathcal M(\mathcal E)}^{\mathrm{SOT}}
=
\mathcal M(L^\infty(\Omega)),
\]
where the closures are taken with respect to the weak and strong operator topologies, respectively.
\end{prop}
\begin{proof}
Let $a\in L^\infty(\Omega)$.
Note that
$\mathcal E$ is $\sigma(L^\infty(\Omega),L^1(\Omega))$-dense in
$L^\infty(\Omega)$ (see, for example, \cite[Theorem 2.2]{HerreraYanezMaximenkoRamosVazquez2022}).

Hence, let $(a_\alpha)$ be a net in
$\mathcal E$ such that $a_\alpha\to a$ in
$\sigma(L^\infty(\Omega),L^1(\Omega))$. Then for every
$f,g\in L^2(\Omega,\widehat{\nu})$ we have
\[
\langle M_{a_\alpha}f,g\rangle
=
\int_{\Omega}a_\alpha(\xi)f(\xi)\overline{g(\xi)}\,d\widehat{\nu}(\xi)
\to
\int_{\Omega}a(\xi)f(\xi)\overline{g(\xi)}\,d\widehat{\nu}(\xi)
=
\langle M_af,g\rangle,
\]
because $f\overline g\in L^1(\Omega,\widehat{\nu})$. Hence
$M_{a_\alpha}\to M_a$ in the weak operator topology.

Since $\mathcal M(L^\infty(\Omega))$ is WOT-closed and
$\mathcal M(\mathcal E)$ is convex, it follows from \cite[Theorem 5.1.2]{KadisonRingroseI1983} the WOT and SOT closures of
$\mathcal M(\mathcal E)$ coincide.
\end{proof}

Next, recall that a closed densely defined operator $T$ is said to be \emph{affiliated} with a von Neumann algebra $\mathfrak A$ if for every unitary operator $U$ in its commutant $\mathfrak A'$ it holds $U\operatorname{Dom}(T)\subset\operatorname{Dom}(T)$ and
$UTU^*=T$, where the equality holds on $\operatorname{Dom}(T)$.

Notice that $\mathcal M(L^\infty(\Omega))$ is a maximal abelian von Neumann algebra with $\mathcal M(L^\infty(\Omega)) = \mathcal M(L^\infty(\Omega))'$. Further, every unitary operator in $\mathcal M(L^\infty(\Omega))$ is of the form $M_u$, for some $u\in L^\infty(\Omega)$ with $|u|=1$ almost everywhere.
We will make use of the following important result, adapted to our setting.
\begin{prop}[{\cite[Theorem 5.6.4]{KadisonRingroseI1983}}]
\label{prop:T affiliated is Mg}
    A closed densely defined operator $T$ is affiliated with the von Neumann algebra $\mathcal M(L^\infty(\Omega))$ if and only if $T=M_g$, for some measurable almost everywhere finite function $g\colon \Omega\to\bC$.
\end{prop}

\begin{thm}
\label{thm:unbdd trans inv diagonalization}
Let $S\colon \cD_0\longrightarrow H$
be a closable (possibly) unbounded operator such that
\begin{equation}
    \label{eq:rho S=S rho unbounded}
\rho(x)S=S\rho(x),\qquad \forall x\in G.
\end{equation}
Then there is a unique, up to equality almost everywhere, measurable and almost everywhere finite function $b\colon \Omega\to \bC$ such that
\[
RS R^*=M_b,
\]
where equality holds on the domain $R(\cD_0)$.
\end{thm}
\begin{proof}
    Let $T=RS R^*$. By the unitarity of $R$, $T$ is closable. We denote its closure by $\overline{T}$.
    
    As we showed above, the space $\mathcal M(\mathcal E)$ is SOT dense in $\mathcal M(L^\infty(\Omega))$.
    Thus, let $a\in L^\infty(\Omega)$ and $(a_\alpha)_\alpha$ be a net in $\mathcal E$ such that $M_{a_\alpha}$ converges strongly to $M_a$. In particular, for a fixed $f\in R(\cD_0)$, one has
    \[
    \|M_{a_{\alpha}}Tf-M_a(Tf)\|^2
    +
    \|M_{a_{\alpha}}f-M_af\|^2\to0,
    \]
    as $\alpha$ runs along the directed set.    
    By \eqref{eq:rho S=S rho unbounded} and \eqref{eq:Rrho x R* equals Ex}, for all $\alpha$ we have
    \[
    M_{a_{\alpha}}Tf=TM_{a_{\alpha}}f,
    \]
    and hence
    \[
    (M_{a_{\alpha}}f,M_{a_{\alpha}}Tf)=(M_{a_{\alpha}}f,TM_{a_{\alpha}}f)\in\operatorname{Graph}(T).
    \]
    Therefore,
    \[
    (M_af,M_aTf)=\lim_\alpha
    (M_{a_{\alpha}}f,TM_{a_{\alpha}}f)\in \operatorname{Graph}(\overline{T}).
    \]
    This implies that $M_af\in \operatorname{Dom}(\overline{T})$ and
    \begin{equation}
        \label{eq:MaTf=clTMaf}
        M_aTf=\overline{T}M_af,\qquad f\in R(\cD_0).
    \end{equation}
    
    Let now $(f,\overline{T}f)\in\operatorname{Graph}(\overline{T})$ and let $((f_\beta,Tf_\beta))_\beta$ be a net in $\operatorname{Graph}(T)$ converging to $(f,\overline{T}f)$.
    Boundedness of $M_a$ and \eqref{eq:MaTf=clTMaf} imply
    \[
    (M_af_\beta,\overline{T}M_af_\beta)=
    (M_af_\beta,M_a\overline{T}f_\beta)\to (M_af,M_a\overline{T}f),
    \]
    therefore we obtain
    \[
    M_a\overline{T}=\overline{T}M_a,
    \]
    on the domain $\operatorname{Dom}(\overline{T})$.
    
    In particular, the operator $\overline{T}$ is affiliated to $\mathcal M(L^\infty(\Omega))$ and by Proposition \ref{prop:T affiliated is Mg} there is a measurable almost everywhere finite function $b$ such that $\overline{T}=M_{b}$ on $\operatorname{Dom}(\overline{T})$. The result follows by restricting back to $\cD_0$.
\end{proof}

\begin{cor}
    \label{cor:Spsi is Mbpsi}
    Let $\psi\in \cA_0$. Then it holds on $R(\cD_0)$:
    \[
    RS_\psi R^*=M_{b_\psi},
    \]
    where
    \begin{equation}
\label{eq:symbol b unbounded S_psi}
    b_\psi(\xi) = \frac{(R\psi(\cdot,\cdot,y))(\xi)}{\overline{q_\xi(y)}},
    \qquad \xi\in \Omega_y,\ y\in Y_0.
\end{equation}
    Here, $Y_0$ is chosen as in Proposition \ref{prop:partition Y0}.
\end{cor}
\begin{proof}
    As was already noted, $S_\psi$ is translation invariant. Moreover, Proposition \ref{prop:tilde S adjoint closable} shows that it is closable. Hence, Theorem \ref{thm:unbdd trans inv diagonalization} shows that $RS_\psi R^*=M_{b_\psi}$ for some $b_\psi$.
    By Proposition \ref{prop:RK=q}, one has
\[
R\psi(\cdot,\cdot,y)
=R\widetilde S_\psi K_{0,y}
=M_{b_\psi}\overline{q_\bullet(y)}
=b_\psi\overline{q_\bullet(y)}.
\]
Dividing on each element $\Omega_y$ of the partition gives \eqref{eq:symbol b unbounded S_psi}.
\end{proof}

\subsection{\texorpdfstring{The space of symbols $\cF_0$}{The space of symbols F0}}
Let $\mathcal F_0$ be the space of all functions $b_\psi\colon \Omega\to \bC$ with $RS_\psi R^*=M_{b_\psi}$ and $\psi\in\mathcal A_0$.
The space $\mathcal F_0$ can be characterized as follows:
\begin{prop}
\label{prop:characterization F0}
Let $b\colon \Omega\to\bC$ be a measurable function. The following conditions are equivalent:
\begin{enumerate}[label=(\arabic*)]
    \item $b\in \cF_0$.
    \item The multiplication operator $M_b$ is well-defined on
    \[
    R(\cD_0)=\operatorname{span}\{E_{-x}\overline{q_\bullet(y)}\colon x\in G,\ y\in Y\}.
    \]
    \item For every $y\in Y$ it holds
    \[
\int_{\Omega}|b(\xi)|^2|q_\xi(y)|^2 d\widehat{\nu}(\xi)<\infty.
\]
\end{enumerate}
In particular,
\[
\cF_0=\bigcap_{y\in Y} L^2(\Omega,\mu_y),
\]
where $\mu_y$ is given by \eqref{eq:definition muy}.
\end{prop}
\begin{proof}
The equivalence of (2) and (3) is clear. We show the equivalence of (1) and (2).

$(1)\implies (2)$\quad Let $b\in\mathcal F_0$. By definition there is $\psi\in \cA_0$ such that $b=b_\psi$ and $RS_\psi R^*=M_{b_\psi}$. Since the operator $S_\psi$ is well-defined on the common domain $\cD_0$, the operator $M_{b_\psi}$ is in turn well-defined on the domain $R(\cD_0)$.

    $(2)\implies (1)$\quad Conversely, let $b$ be a measurable function satisfying the above condition. Then the operator $M_b$ is well-defined on $R(\cD_0)$. 
    Therefore, $R^*M_{b}R$ is well-defined on $\cD_0$.
    
    Define the function $\psi\colon G\times Y\times Y\to\bC$ given by
    \[
    \psi(x,y,v)=
    R^*\big(\overline{q_\bullet(v)}b\big)(x,y).
    \]
    Note that $\psi(\cdot,\cdot,y)\in H$. An easy computation (see \cite[Lemma 4.4]{BaisMaximenkoVenkuNaidu2025}) using \eqref{eq:def Rstar}  shows that also $\overline{\psi(-\cdot,y,\cdot)}\in H$, whence $\psi\in\cA_0$.
    
    Finally, it follows from \eqref{eq:tildeSK is psi} that
    \begin{align*}
    S_\psi K_{x,y}&=\psi(\cdot-x,\cdot,y)
    =
    \rho(x)
    R^*\big(\overline{q_\bullet(y)}b\big)\\
    &=
    R^*\big(M_bE_{-x}\overline{q_\bullet(y)}\big)
    =
    R^*M_bRK_{x,y},
    \end{align*}
    which shows that $R^*M_bR=S_\psi$.
\end{proof}

\begin{cor}
\label{cor:equiv S0 A0 F0}
    Let $S$ be a (possibly) unbounded operator defined on $\cD_0$.
    The following conditions are equivalent:
    \begin{enumerate}[label=(\arabic*)]
        \item $S$ is a translation-invariant closable operator.
        \item There is $\psi\in\cA_0$ such that $S=S_\psi$. That is, $S\in \cS_0$.
        \item It holds $RSR^*=M_b$ for some $b\in \cF_0$.
    \end{enumerate}
    In particular,
    \[
    \cS_0\cap\cB(H)= \cC(\rho)=\cS_\infty.
    \]
\end{cor}
\begin{proof}
That $(2)$ implies $(3)$ is in essence the content of Corollary \ref{cor:Spsi is Mbpsi}. On the other hand, $(3)$ implies $(1)$ since for all $x\in G$ it holds $M_{E_{-x}}M_b=M_bM_{E_{-x}}$ and thus
\[
\rho(x)S=S\rho(x).
\]
The operator $S$ is closable since $M_b$ is closable, as $M_b\subset M_b^{\text{max}}$.

    It remains to show that $(1)$ implies $(2)$. Let $S$ be closable and translation invariant. By Theorem \ref{thm:unbdd trans inv diagonalization}, we have $RSR^*=M_b$ for some measurable a.e. finite $b$. Proposition \ref{prop:characterization F0} implies that $b\in \cF_0$ and hence there is some $\psi\in \cA_0$ such that $b=b_\psi$ and $S=R^*M_bR=R^*M_{b_\psi} R=S_\psi$.    
\end{proof}
Since a multiplication operator is bounded if and only if its symbol is essentially bounded, we obtain
\[
S_\psi\in \cS_\infty \iff b_\psi\in L^\infty(\Omega).
\]
Therefore, we define
\[
\cF_\infty=L^\infty(\Omega).
\]
We obviously have $\cF_\infty\subset\cF_0$.

On the other hand, note that Proposition \ref{prop:characterization F0} suggests defining a natural topology on $\mathcal F_0$.
For each $y\in Y$ consider the seminorm
\[
\| \cdot \|_{y}\colon \mathcal F_0\longrightarrow [0,\infty),
\]
given by the norm in the space $L^2(\Omega,\mu_y)$. That is,
\[
\| a \|_{y}
=
\left(
\int_{\Omega}|a(\xi)|^2|q_\xi(y)|^2d\widehat{\nu}(\xi)
\right)^{\frac{1}{2}},\quad a\in \mathcal F_0.
\]
We endow $\mathcal F_0$ with the initial topology generated by the family of seminorms $(\|\cdot\|_y)_{y\in Y}$ and denote this topology by $\tau_{\cF_0}$.

Recall that a Cauchy net in a topological vector space is a net $(x_\alpha)_\alpha$ such that for every neighborhood $U$ there is some $\alpha_0$ such that $x_{\alpha_1}-x_{\alpha_2}\in U$ for every $\alpha_1,\alpha_2\succeq \alpha_0$. A topological vector space is \emph{complete} if every Cauchy net is convergent. If the topology is induced by a pseudometric, then this coincides with the usual notion using sequences (sequential completeness). This holds, for example, in a Fréchet space. See \cite{NariciBeckenstein2010,treves2006topological} for more details.
\begin{prop}
The space $(\mathcal F_0,\tau_{\cF_0})$ is a complete locally convex topological vector space.
\end{prop}
\begin{proof}
Since the topology $\tau_{\cF_0}$ is generated by seminorms, the space is locally convex. Moreover, using the seminorms $\|\cdot\|_y$ associated to $y\in Y_0$, where $Y_0$ is given by the partition in Proposition \ref{prop:partition Y0}, one can see that $\cF_0$ is Hausdorff. We only need to prove completeness.

Let $(a_\alpha)$ be a Cauchy net in $\cF_0$.
For every $y\in Y$, the net $(a_\alpha)$ is Cauchy in the Hilbert space $L^2(\Omega,\mu_y)$. Hence there is an element $a^y$ of this space such that
\[
\|a_\alpha-a^y\|_y\to 0.
\]
We now take the countable partition $(\Omega_y)_{y\in Y_0}$ from Proposition \ref{prop:partition Y0} and define a measurable function $a$ on $\Omega$ by
\[
a|_{\Omega_y}=a^y|_{\Omega_y},\qquad y\in Y_0.
\]
This definition is harmless after modifying the representatives of the functions $a^y$ on null sets.

It remains to check that $a$ is the limit of $(a_\alpha)$ for every seminorm $\|\cdot\|_y$. Fix $y\in Y$. On each set
\[
\Omega_{y_0}\cap\{ |q_\bullet(y_0)|\ge 1/k\}\cap\{ |q_\bullet(y)|\ge 1/n\},
\qquad y_0\in Y_0,\ k,n\in\bN,
\]
convergence in both weighted $L^2$-spaces implies convergence in the ordinary $L^2$-space with respect to $\widehat\nu$. Hence the two limits $a^{y_0}$ and $a^y$ agree there. Since these sets exhaust $\Omega_{y_0}\cap\{|q_\bullet(y)|>0\}$ up to null sets, we get $a=a^y$ $\mu_y$-almost everywhere. Therefore
\[
\|a_\alpha-a\|_y=\|a_\alpha-a^y\|_y\to 0,
\]
for every $y\in Y$, which proves that $a_\alpha\to a$ in $\tau_{\cF_0}$.
\end{proof}

\begin{cor}
    \label{cor:F0 frechet countable seminorms}
    $(\cF_0,\tau_{\cF_0})$ is a Fréchet space if and only if its topology is generated by a countable subset of $(\|\cdot\|_y)_{y\in Y}$.
\end{cor}

\section{Commutative topological algebras}
\label{sec:topological algebras}

\subsection{\texorpdfstring{The algebra $\cF_M$}{The algebra FM}}

The product of two elements in $\cF_0$ is well-defined as the usual point-wise product. However, as one might expect, the space $\mathcal F_0$ is in general not closed under multiplication. Indeed, if $a,b\in \mathcal F_0$, then by Proposition \ref{prop:characterization F0} we have $ab\in \mathcal F_0$ only when $\|ab\|_y<\infty$ for all $y\in Y$.
If this is the case, the product is obviously commutative.
This gives immediately the following result.
\begin{prop}
\label{prop:product in F0 fixed a}
Given a fixed $a\in \cF_0$ one has $ab\in \cF_0$ if and only if
    \[
    b\in L^2(\Omega,|a|^2|q_\bullet(y)|^2d\widehat{\nu}),\qquad\forall y\in Y.
    \]
    In particular, it holds $ab\in\cF_0$ if $a,b\in \cF_0$ and
    $a\in L^\infty(\Omega)$ or $b\in L^\infty(\Omega)$.
\end{prop}

Generally, it is easy to construct examples of functions $a,b,c\in \mathcal F_0$ such that $ab\in \mathcal F_0$ but $ac\notin \mathcal F_0$. A much more subtle question is how large can one form an algebra inside $\cF_0$ with this product.

Define
\begin{equation}
\label{eq:def FM powers}
\mathcal F_M
=
\{a\in\mathcal F_0\colon a^m\in\mathcal F_0\text{ for every }m\in\bN\}.
\end{equation}
This space is an algebra as follows from the simple estimate
\begin{equation}
    \label{eq:FM product ineq} 
    \|ab\|_{y,m}\leq \|a\|_{y,2m}\|b\|_{y,2m}.
\end{equation}
The following proposition shows that this is precisely the largest subalgebra of $\mathcal F_0$ with respect to pointwise multiplication.
\begin{prop}
\label{prop:characterization FM}
Let $a\colon \Omega\to \bC$ be a measurable function. Then the following conditions are equivalent:
\begin{enumerate}
    \item[(1)] $a\in \mathcal F_M$.
    \item[(2)] $a^m\in \mathcal F_0$ for every $m\in\bN$.
    \item[(3)] $\|a^m\|_y<\infty$ for every $y\in Y$ and $m\in\bN$.
\end{enumerate}
Moreover, $\mathcal F_M$ is an algebra and every subalgebra of $\mathcal F_0$ is contained in $\mathcal F_M$.
\end{prop}
\begin{proof}
The equivalence of (1) and (2) is the definition \eqref{eq:def FM powers}, and Proposition \ref{prop:characterization F0} gives the equivalence of (2) and (3).

Let $a,b\in\mathcal F_M$.  For $r,s\in\bN$ and $y\in Y$, H\"older's inequality gives
\[
\int_\Omega |a(\xi)|^{2r}|b(\xi)|^{2s}|q_\xi(y)|^2d\widehat\nu(\xi)
\le
\|a^{2r}\|_y\,\|b^{2s}\|_y<\infty.
\]
Thus every monomial $a^rb^s$ belongs to $\mathcal F_0$.  Expanding powers of $a+b$ shows $(a+b)^n\in\mathcal F_0$ for all $n$, and the same estimate gives $(ab)^n\in\mathcal F_0$ for all $n$.  Hence $a+b$ and $ab$ belong to $\mathcal F_M$.

Finally, if $\mathcal B\subset\mathcal F_0$ is any algebra and $a\in\mathcal B$, then all powers $a^m$ belong to $\mathcal B\subset\mathcal F_0$.  Therefore $a\in\mathcal F_M$, proving $\mathcal B\subset\mathcal F_M$.
\end{proof}

The above proposition shows that
\[
\mathcal F_M=\bigcap_{y\in Y, m\in \bN} L^{2m}(\Omega,\mu_y).
\]
Consider thus the family of seminorms
\[
a\mapsto \|a\|_{y,m}\coloneqq(\|a^m\|_y)^{1/m},
\]
for $y\in Y$ and $m\in \bN$, and endow $\mathcal F_M$ with the initial topology $\tau_{\cF_M}$ generated by them.
Furthermore, we can endow $\mathcal F_M$ with the natural involution
\[
*\colon \mathcal F_M\longrightarrow \mathcal F_M,\qquad 
a\longmapsto \overline{a}.
\]
This involution is obviously continuous with respect to the topology in $\mathcal F_M$ and extends the canonical involution in $\cF_\infty=L^\infty(\Omega)$. Note, however, that it is defined on functions rather than on the corresponding multiplication operators. In terms of operators, one can equivalently define the involution through the rule
\begin{equation}
\label{eq:def invol in F}
\mathcal F_M\ni a\mapsto M_a\mapsto M_a^*\mapsto M_a^*|_{R(\cD_0)}=M_{\overline{a}}|_{R(\cD_0)}\mapsto \overline{a}\in \mathcal F_M.
\end{equation}
This will come in handy when translating the involution to operators.

\begin{rem}
\label{rem:arens algebras}
Note that, for every \emph{fixed} $y\in Y$, the space
\[
L^\omega(\Omega,\mu_y)
\coloneqq
\bigcap_{m\in \bN}
L^{2m}(\Omega,\mu_y)
=
\bigcap_{p\geq 1}
L^{p}(\Omega,\mu_y)
\]
is already a complete locally convex topological algebra with respect to the initial topology generated by the family $(\|\cdot\|_{y,m})_{m\in\bN}$. These algebras are well known in the literature and are called Arens algebras (see, for example, \cite{Arens1946,AbdullaevChilin1998}).
\end{rem}

\begin{prop}
The algebra $\mathcal F_M$ endowed with the initial topology $\tau_{\cF_M}$ is a complete locally convex topological $*$-algebra and is continuously included (as a topological vector space) in $\mathcal F_0$.

Furthermore, the algebra $\cF_\infty=L^\infty(\Omega)$, endowed with its usual $L^\infty(\Omega,\widehat\nu)$ norm topology, is continuously included in $\mathcal F_M$ and for every $f\in \cF_\infty$ one has
\[
\|f\|_\infty =
\sup_{y\in Y_0}\lim_{m\to\infty}\|f\|_{y,m},
\]
where $\|\cdot\|_\infty$ denotes the essential supremum with respect to $\widehat\nu$ and $Y_0$ is associated to the partition $(\Omega_y)_{y\in Y_0}$ from Proposition \ref{prop:partition Y0}.
\end{prop}
\begin{proof}
The product in $\cF_M$ is continuous because of \eqref{eq:FM product ineq}. Let now $(a_\alpha)$ be a Cauchy net in $\cF_M$. Since the seminorms with $m=1$ are precisely the seminorms defining $\tau_{\cF_0}$, the net is Cauchy in $\cF_0$. By completeness of $\cF_0$, there exists $a\in\cF_0$ such that
\[
\|a_\alpha-a\|_y\to 0,\qquad y\in Y.
\]
Fix $y\in Y$ and $m\in\bN$. The same net is Cauchy in the Banach space $L^{2m}(\Omega,\mu_y)$, because
\[
\|a_\alpha-a_\beta\|_{L^{2m}(\Omega,\mu_y)}=\|a_\alpha-a_\beta\|_{y,m}.
\]
Therefore there is $c_{y,m}\in L^{2m}(\Omega,\mu_y)$ such that $a_\alpha\to c_{y,m}$ in $L^{2m}(\Omega,\mu_y)$. Since $\mu_y(\Omega)<\infty$, convergence in either $L^2(\mu_y)$ or $L^{2m}(\mu_y)$ implies convergence in $\mu_y$-measure. The limit in measure is unique, so $c_{y,m}=a$ $\mu_y$-almost everywhere. Consequently $a\in L^{2m}(\Omega,\mu_y)$ and
\[
\|a_\alpha-a\|_{y,m}
=
\left(\int_\Omega |a_\alpha-a|^{2m}\,d\mu_y\right)^{1/(2m)}
\to 0.
\]
Since this holds for every $y\in Y$ and $m\in\bN$, we have $a\in\cF_M$ and $a_\alpha\to a$ in $\tau_{\cF_M}$. This proves completeness.

On the other hand, let $f\in \cF_\infty=L^\infty(\Omega)$. Then for every $y\in Y$ and $m\in\bN$ one has
\[
\|f\|_{y,m}^{2m}=\int_{\Omega}|f(\xi)|^{2m}|q_\xi(y)|^2d\widehat{\nu}(\xi)\leq \|f\|_\infty^{2m}\|q_{\bullet}(y)\|^2_{L^2(\Omega,\widehat{\nu})},
\]
showing that the inclusion is continuous. Moreover, by the same estimate, for every $y\in Y$ one has $f\in \bigcap_{m\in\bN}L^{2m}(\Omega,|q_\bullet(y)|^2d\widehat{\nu})$. Since $\mu_y$ is finite, the standard formula for $L^p$ norms on finite measure spaces gives
\[
\|f|_{U_y}\|_{L^\infty(\mu_y)}=\lim_{m\to\infty}\|f\|_{y,m}.
\]
For $y\in Y_0$, the measures $\mu_y$ and $\widehat\nu$ are equivalent on $\Omega_y$, so this limit is also the $\widehat\nu$-essential supremum of $f$ on $\Omega_y$. The formula for $\|f\|_\infty$ follows from the fact that $\Omega=\bigcup_{y\in Y_0}\Omega_y$.
\end{proof}

\begin{cor}
\label{cor:F0 frechet then FM frechet}
If $(\cF_0,\tau_{\cF_0})$ is a Fréchet space, then
$\mathcal F_M$ is a Fréchet $*$-algebra.
\end{cor}
\begin{proof}
If $\cF_0$ is Fr\'echet, then by Corollary \ref{cor:F0 frechet countable seminorms} its topology is generated by a countable family $(\|\cdot\|_{y_k})_{k\in\bN}$.  For every $y\in Y$ there are indices $k_1,\ldots,k_N$ and a constant $C>0$ such that
$\|f\|_y\le C\sum_{j=1}^N\|f\|_{y_{k_j}}$ for all $f\in\cF_0$. Applying this inequality to $f=a^m$, with $a\in\cF_M$, gives
\[
\|a\|_{y,m}^m=\|a^m\|_y
\le C\sum_{j=1}^N\|a^m\|_{y_{k_j}}
=C\sum_{j=1}^N\|a\|_{y_{k_j},m}^m.
\]
Taking $m$-th roots and using the elementary inequality $(x_1+\cdots+x_N)^{1/m}\le x_1^{1/m}+\cdots+x_N^{1/m}$ gives a constant $C_m>0$ such that
\[
\|a\|_{y,m}\le C_m\sum_{j=1}^N\|a\|_{y_{k_j},m}.
\]
Thus the topology of $\cF_M$ is generated by the countable family
$(\|\cdot\|_{y_k,m})_{k,m\in\bN}$.
\end{proof}

\subsection{\texorpdfstring{On the inclusions $\cF_\infty\subset\cF_M\subset\cF_0$}{On the inclusion Finfty subset FM subset F0}}
\label{sec:inclusions F FM F0}

An interesting question, which is closely related to the extension of the product $\odot$ in $\cA_\infty$ (see Corollary \ref{cor:extension odot}) is whether the inclusions $\cF_\infty\subset \cF_M$ and $\cF_M\subset \cF_0$ are strict.

We begin by studying the inclusion $\cF_M\subset \cF_0$. We were not able to find a complete characterization of the equality $\cF_0=\cF_M$ in terms of the measures $\mu_y$, $y\in Y$, and we leave this as an open question. Nevertheless, we present several partial results that cover many important examples and shed some light on the underlying phenomena.

\begin{prop}
\label{prop:atomicity characterization}
    The following statements are equivalent: 
    \begin{enumerate}[label=(\arabic*)]
        \item $\mu_y$ is an atomic measure for every $y\in Y$.
    
        \item $\mu_y|_U$ is a non zero atomic measure for some $y\in Y$ and some open $U\subset\Omega$.

        \item The dual Haar measure $\widehat{\nu}$ is atomic.

        \item The dual group $\widehat{G}$ is discrete.

        \item The group $G$ is compact.
    \end{enumerate}
\end{prop}
\begin{proof}   
    Note first that the equivalence of (3), (4) and (5) follows easily from the fact that the Haar measure is a Radon measure and from Pontryagin's duality.
    
    On the other hand, since $\mu_y\ll \widehat{\nu}$, (3) implies (1), and (1) trivially implies (2).    
    
    Therefore, we only need to prove that (2) implies (3). Let thus $y\in Y$ and $U\subset \Omega$ open be such that $\mu_y|_U$ is atomic and non-zero.
    Since the measure $\mu_y|_U$ is non-zero, we have $U_y\cap U\neq\emptyset$, and since $q_\bullet(y)\neq 0$ on $U_y\cap U$, the measures    $\widehat{\nu}|_{U_y\cap U}$ and $\mu_y|_{U_y\cap U}$ are equivalent.
    
    Since $\mu_y|_U$ is atomic, then $\mu_y|_{U_y\cap U}$ is also atomic and, therefore, $\widehat{\nu}|_{U_y\cap U}$ is atomic as well.
    Finally, the invariance of the Haar measure implies that $\widehat{\nu}$ is itself atomic.
    If a Haar measure on a locally compact group has an atom, then the singleton containing that atom has positive measure; by translation invariance every singleton has the same positive measure, hence the group is discrete.
\end{proof}

\begin{thm}
\label{thm:particular characterization A0=AM}
Suppose that $(\cF_0,\tau_{\cF_0})$ is a Fréchet space.
Then the condition $\cF_M=\cF_0$ implies that $G$ is compact.
\end{thm}
\begin{proof}
We reason similarly to \cite[Theorem 1]{Villani1985}.
By Corollary \ref{cor:F0 frechet then FM frechet}, the algebra $\cF_M$ is a Fréchet $*$-algebra.
Let $m\in\bN$ with $m>1$ and fix $y_0\in Y$ such that $\mu_{y_0}\neq0$. Then $\cF_M=\cF_0$ implies the set-theoretic inclusion
$\cF_0\subset L^{2m}(\Omega,\mu_{y_0})$.

By passing to an almost everywhere convergent subsequence of an $L^2$-convergent sequence, one readily checks that the inclusion has a closed graph (note that the Fréchet assumption allows us to argue sequentially) and by the closed graph theorem for Fréchet spaces (see, for example, \cite[Theorem 2.15]{rudin1991functional}), we conclude that the inclusion is continuous.

Hence, there are $y_1,\ldots,y_L\in Y$ and a constant $C_1>0$ such that
\begin{equation}
\label{eq:continuous inclusion L2 in L2m inequality}
\|f\|_{y_0,m}\leq C_1(\|f\|_{y_1}+\cdots+\|f\|_{y_L}),
\end{equation}
for all $f\in \cF_0$. 
Note that $\mu_{y_0}(U_{y_\ell})>0$ for at least one $y_\ell$ with $\ell\in\{1,\ldots,L\}$. Indeed, otherwise letting
\[
A=
\Omega\backslash\left[\bigcup_{\ell=1}^LU_{y_\ell}\right],
\]
one would have $\|\chi_A\|_{y_0,m}>0$ but $\|\chi_A\|_{y_\ell}=0$, $\ell=1,\ldots,L$, which contradicts \eqref{eq:continuous inclusion L2 in L2m inequality}.
Fix thus some $\ell'\in\{1,\ldots,L\}$ with $\mu_{y_0}(U_{y_{\ell'}})>0$. 

Let $K\subset U_{y_0}\cap U_{y_{\ell'}}$ be any compact set with $\widehat{\nu}(K)>0$ (which exists by regularity of the Haar measure). By Corollary \ref{cor:|q| continuous in xi}, the functions $|q_\bullet(y_\ell)|$
are continuous for all $\ell=1,\ldots,L$ and, by construction, $|q_\bullet(y_0)|$ and $|q_\bullet(y_{\ell'})|$ are strictly positive on $U_{y_0}\cap U_{y_{\ell'}}$.

Thus there are constants $A_\ell,B_\ell\geq 0$, $\ell=0,1,\ldots,L$ such that 
\[
A_0,A_{\ell'},B_0,B_{\ell'}>0
\]
and
\[
A_\ell\leq |q_\xi(y_\ell)|\leq B_\ell,\qquad\forall \xi\in K,\ \ell=0,1,\ldots,L.
\]
Apply \eqref{eq:continuous inclusion L2 in L2m inequality} to functions $f\in\cF_0$ supported in $K$. From the preceding estimates we obtain a constant $C_2>0$ such that
\[
\left(
\int_{K}|f(\xi)|^{2m}d\widehat{\nu}(\xi)\right)^{\frac{1}{2m}}
\leq C_2
\left(
\int_{K}|f(\xi)|^2d\widehat{\nu}(\xi)
\right)^{\frac{1}{2}}.
\]
Since $\chi_{E}\in \cF_\infty\subset \bigcap_{y\in Y} L^2(\Omega,\mu_y)$ for every measurable set $E\subset K$ with positive measure $\widehat{\nu}(E)>0$, this implies that for such a set
\[
\widehat{\nu}(E)^{\frac{1}{2m}}\leq C_2\widehat{\nu}(E)^{\frac{1}{2}}.
\]
Therefore,
\[
\widehat{\nu}(E)\geq C_3,
\]
for some constant $C_3>0$ and every measurable $E\subset K$ with positive measure.
It follows from this that $\widehat{\nu}|_K$ is atomic.
Similarly to Proposition \ref{prop:atomicity characterization}, it follows that $\widehat\nu$ is atomic as well and, hence, that $G$ is compact.
\end{proof}

Surprisingly, there are non-trivial examples with $G$ compact for both cases $\cF_M\subsetneq\cF_0$ and $\cF_M=\cF_0$ (see Section \ref{sec:examples}).

We turn now to the analysis of the equality $\cF_\infty=\cF_M$.
\begin{prop}
\label{prop:F=F0 iff F=FM}
It holds $\cF_\infty=\cF_0$ if and only if $\cF_\infty=\cF_M$.
\end{prop}
\begin{proof}
    We only need to prove the sufficiency. Suppose that $\cF_\infty\neq \cF_0$.    
    This means that there is a function
    \[
    a\in \bigcap_{y\in Y}L^2(\Omega,\mu_y)\backslash L^\infty(\Omega).
    \]
    Since $a$ is unbounded, the function $\log(|a|+1)$ is also unbounded and for some constants $C_m>0$ one has
    \[
    (\log(|a|+1))^{2m}\leq C_m(1+|a|^2),\qquad \forall m\in \bN,
    \]
    This implies that
    $\log(|a|+1)\in 
    \cF_M
    \backslash \cF_\infty$.
\end{proof}

\begin{cor}
\label{cor:A=AM implies compactness}
Suppose that $(\cF_0,\tau_{\cF_0})$ is a Fréchet space. Then the condition $\cF_\infty=\cF_M$ implies that $G$ is compact.
\end{cor}
\begin{proof}
    By Proposition \ref{prop:F=F0 iff F=FM}, the condition $\cF_\infty=\cF_M$ implies $\cF_\infty=\cF_0$, whence $\cF_M=\cF_0$. Then Theorem \ref{thm:particular characterization A0=AM} implies that $G$ is compact.
\end{proof}

\begin{cor}
\label{cor:G not compact both inclusions strict}
Suppose that $(\cF_0,\tau_{\cF_0})$ is a Fréchet space.
    If $G$ is not compact, then both inclusions $\cF_\infty\subset \cF_M$ and $\cF_M\subset \cF_0$ are strict.
\end{cor}
\begin{proof}
It follows from Theorem \ref{thm:particular characterization A0=AM} and Corollary \ref{cor:A=AM implies compactness}.
\end{proof}

\begin{rem}
    The equality $\cF_\infty=\cF_0$ seems to be false for most cases (see Section \ref{sec:examples}). Nevertheless, we want to note that, at the same time, the equality depends on how strong the restriction imposed by the intersection along $y\in Y$ is.
    To illustrate this, we consider two extreme cases.
    For both examples, we suppose that $\widehat{G}=\bZ$ with $\widehat{\nu}$ the counting measure and that $\Omega=\bN$.
    \begin{itemize}
        \item Suppose first that there is only one point $Y=\{y_0\}$ and thus $\Omega=U_y=\bN$.
        Let $q_\bullet(y_0)=(g_k)_{k\in\bN}$ be any sequence in $\ell^2(\bN)$. Then, there is a subsequence $(k_j)_{j\in \bN}$ such that $|g_{k_j}|^2<1/3^j$ for all $j\in \bN$.
        
        Let $a=(a_k)_{k\in\bN}$ be defined by
        \[
        a_k=
        \begin{cases}
            2^{j/2},&\qquad\text{if }k=k_j,\\
            0,&\qquad \text{otherwise}.
        \end{cases}
        \]
        Then $a$ is unbounded but
        \[
        \|a\|_{y_0}^2=
        \sum_{k=1}^\infty
        |a_kg_k|^2
        <
        \sum_{j=1}^\infty
        \left(\frac{2}{3}\right)^j<\infty,
        \]
        which means that $a\in \cF_0\backslash \cF_\infty$. By Proposition \ref{prop:F=F0 iff F=FM}, we conclude that $\cF_\infty\neq\cF_M$.

        \item On the other hand, let $Y=\ell^2(\bN)$ and $q_\bullet(y)=y$ for every $y\in Y$.
        Let $a\in \bigcap_{y\in Y}L^2(\Omega,\mu_y)$. This means that, for every sequence $(g_k)\in \ell^2(\bN)$ one has $(a_kg_k)\in \ell^2(\bN)$.
        
        Therefore, the multiplication operator $M_a$ with symbol $a$ is defined on all of $\ell^2(\bN)$. This multiplication operator is closed; hence the closed graph theorem implies that $M_a$ is bounded. Consequently $a\in L^\infty(\bN)$, and therefore \[ \cF_\infty=\cF_M=\cF_0.\]
        This example is admittedly somewhat artificial, since $Y=\ell^2(\bN)$ is not a $\sigma$-finite measure space and the Fréchet assumption is probably false.
        Nevertheless, it illustrates that intersecting ``too many'' spaces of the form $L^2(\Omega,\mu_y)$ forces the equality $\cF_\infty=\cF_M=\cF_0$.
    \end{itemize}
\end{rem}

\section{Topological algebras of operators and integral kernels}
\label{sec:algebras A S}

In this section we carry over the results from the previous section to corresponding algebras of operators and to algebras of kernel functions on $H$.

\subsection{The operator side}

Consider the space $\cS_0$. In Corollary \ref{cor:equiv S0 A0 F0} we showed that it is exactly the space of all closable translation-invariant operators defined on $\cD_0$.
Moreover, given $S\in\cS_0$ one has $RSR^*=M_b$ for a unique (up to a zero measure set) $b\in\cF_0$.

Since the operator $R\colon H\to L^2(\Omega,d\widehat{\nu})$ is unitary, for every $y\in Y$ one has
\begin{align*}
    \|S K_{0,y}\|_H
    &=\|M_{b}q_{\bullet}(y)\|_{L^2(\Omega)}
    =\|b \|_{y}.
\end{align*}
Thus, we endow $\mathcal S_0$ with the initial topology $\tau_{\mathcal S_0}$ generated by the seminorms
\[
S\longmapsto \|S K_{0,y}\|_H,
\qquad y\in Y.
\]

Let now $\cS_M$ be the subspace of $\cS_0$ consisting of all operators $S_{\psi_a}$ with $a\in \mathcal F_M$. 
If $S\in\mathcal S_0$ has symbol $a\in\mathcal F_0$, then
$S\in\mathcal S_M$ if and only if $a^m\in\mathcal F_0$ for every
$m\in\mathbb N$, equivalently if and only if
\[
\|R^*M_{a^m}RK_{0,y}\|_H<\infty
\]
for all $m\in\mathbb N$ and $y\in Y$.

We define a product on $\cS_M$ by transporting the product from $\cF_M$. That is, we set
\begin{equation}
    \label{eq:product rule SM}
    S_1\star S_2 \coloneqq R^*M_{a_1a_2} R,
\end{equation}
whenever $a_1,a_2\in \cF_M$ with $S_1=R^*M_{a_1}R,S_2=R^*M_{a_2}R\in \cS_0$. By construction, the resulting operator is again in $\cS_M$.
Note that this product essentially coincides with the usual operator composition when both operators are bounded.

We endow $\mathcal S_M$ with the initial topology $\tau_{\cS_M}$ generated by the seminorms
\[
S\longmapsto \|S^{\star m}K_{0,y}\|^{1/m}_H,\qquad y\in Y,\ m\in\bN,
\]
where $S^{1\star}=S$ and $S^{(m+1)\star}=S^{m\star}\star S$, $m\in\bN$.
Furthermore, following \eqref{eq:def invol in F}, we endow $\cS_M$ with the involution $\phantom{\cdot}^\dagger$ defined by
\begin{equation}
    \label{eq:def invol in S}
    S^{\dagger}\coloneqq S^*|_{\cD_0},\qquad S\in \cS_M,
\end{equation}
where $*$ denotes the usual adjoint for unbounded operators. This restriction is well-defined. Indeed, if $S=R^*M_aR|_{\cD_0}$ with $a\in\cF_M$, then $\overline a\in\cF_M$ and $R^*M_{\overline a}R|_{\cD_0}$ is the formal adjoint of $S$ on $\cD_0$. Hence $\cD_0\subset\operatorname{Dom}(S^*)$ and
\[
S^*|_{\cD_0}=R^*M_{\overline a}R|_{\cD_0}.
\]

In general, the product $\star$ should not be confused with the usual composition
of unbounded operators on $\mathcal D_0$, since an operator in
$\mathcal S_M$ need not map $\mathcal D_0$ into itself.
However, while we chose the product $\star$ for transparency, we want to mention that $\cS_M$ does have a realization as an operator algebra under composition, as the following remark shows.

\begin{rem}
\label{rem:S0 is Ostar algebra}
    Recall (see, for example, \cite{Schmudgen1990,Lassner1972}) that an $O^*$-algebra is a family $\mathfrak A$ of closable unbounded operators defined on a common domain $\cD$ such that 
    \begin{enumerate}[label=(\arabic*)]
        \item $I\in \mathfrak A$.
        \item $T\cD\subset \cD$ and $TS\in \mathfrak A$, for all $T,S\in\mathfrak A$.
        \item $\cD\subset \operatorname{Dom}(T^*)$ and $T^*|_{\cD}\in \mathfrak A$ for all $T\in \mathfrak A$.
    \end{enumerate}
    
    Given $a\in \cF_M$ consider its multiplication operator $M_a^{\text{max}}$ with maximal domain
    $\operatorname{Dom}(M_a^{\text{max}})
    =
    \{f\in L^2(\Omega,\widehat{\nu})\colon af\in L^2(\Omega,\widehat{\nu})\}.$
    Set
    \[
    \cD\coloneqq\bigcap_{a\in\cF_M}\operatorname{Dom}(M_a^{\text{max}})
    \]
    and let $\mathfrak A$ be the set of all operators $M_a^{\text{max}}|_{\cD}$, $a\in \cF_M$.
    Note that $\cD$ is a dense subspace of $L^2(\Omega,\widehat{\nu})$ since $R(\cD_0)\subset\cD$.
    We show that $\mathfrak A$ is indeed an $O^*$-algebra isomorphic to $\cS_M$.
    
    Since $I=R^*M_1R$ and $1\in \cF_\infty$, one has $I\in \mathfrak A$.
    Moreover, $(M_a^{\text{max}})^*=M_{\overline{a}}^{\text{max}}$, from which the third condition is also fulfilled.
    As for $(2)$, let $f\in \cD$ and $b\in \cF_M$. Since $\cF_M$ is an algebra one has $ab\in\cF_M$. Thus $f\in \operatorname{Dom}(M_{ab}^{\text{max}})$, which implies $bf\in \operatorname{Dom}(M_{a}^{\text{max}})$.

    Therefore, $\cS_M$ becomes an $O^*$-algebra by transporting the structure from $\mathfrak A$ through the isomorphism 
    \[
    S_\psi\mapsto RS_\psi R^*=M_{b_\psi}\mapsto 
    M_{b_\psi}^{\text{max}}.
    \]
    Note that the domain $\cD$ plays an important role in the structure of $\mathfrak A$ and hence in the structure of $\cS_M$, $\cF_M$ and $\cA_M$. It would be then interesting to give a characterization of it. We leave this as an open question.
\end{rem}

\begin{thm}
\label{thm:statements S0}
The following statements hold.
\begin{enumerate}[label=(\arabic*)]
    \item The space $(\cS_0,\tau_{\cS_0})$ is a complete topological vector space that is isomorphic to $(\cF_0,\tau_{\cF_0})$ via the map
    \[
    S_\psi\longmapsto b_\psi.
    \]

    \item The space $(\cS_M,\tau_{\cS_M})$ is a commutative complete topological $*$-algebra with respect to the operation $\star$ given by \eqref{eq:product rule SM} and the involution $\dagger$ given by \eqref{eq:def invol in S}.
    
    \item The isomorphism in $(1)$ restricts to an isomorphism of topological $*$-algebras making $\cS_M$ isomorphic to $\cF_M$.
    
    \item The product $\star$ extends the operator composition in $\cS_\infty$ and is the largest algebra in the sense that its associated space of symbols $\cF_M$ is the largest algebra in $\cF_0$.
    
    \item The following inclusions hold and are continuous:
    \[
    \cS_\infty\subset \cS_M\subset \cS_0.
    \]
    
    \item For every $S\in \cS_\infty$ one has
    \[
    \|S\|_{\operatorname{op}}=\sup_{y\in Y_0}\lim_{m\to\infty}\|S^{\star m}K_{0,y}\|_H^{1/m}.
    \]
\end{enumerate}
\end{thm}

\begin{proof}
Let $\Gamma:\cS_0\to\cF_0$ be the symbol map, $\Gamma(S)=b$ when $RSR^*=M_b$. Corollary \ref{cor:equiv S0 A0 F0} shows that $\Gamma$ is bijective. Moreover, for every $y\in Y$,
\[
\|S K_{0,y}\|_H=\|\Gamma(S)\|_y.
\]
Thus the topology $\tau_{\cS_0}$ is exactly the pullback of $\tau_{\cF_0}$ under $\Gamma$, and $(\cS_0,\tau_{\cS_0})$ is a complete topological vector space because $(\cF_0,\tau_{\cF_0})$ is complete.

By definition, $\cS_M=\Gamma^{-1}(\cF_M)$, and the seminorms on $\cS_M$ satisfy
\[
\|S^{\star m}K_{0,y}\|_H^{1/m}=\|\Gamma(S)\|_{y,m}.
\]
Also,
\[
\Gamma(S_1\star S_2)=\Gamma(S_1)\Gamma(S_2),
\qquad
\Gamma(S^\dagger)=\overline{\Gamma(S)}.
\]
Hence $\Gamma$ restricts to an isomorphism of topological $*$-algebras from $\cS_M$ onto $\cF_M$. Completeness, commutativity, continuity of multiplication and involution
follow from the corresponding statements for $\cF_M$.

If $S_1,S_2\in\cS_\infty$ have symbols $a_1,a_2\in L^\infty(\Omega)$, then
\[
R(S_1S_2)R^*=M_{a_1}M_{a_2}=M_{a_1a_2},
\]
so $S_1\star S_2=S_1S_2$. The maximality statement follows from Proposition \ref{prop:characterization FM}: every algebra contained in $\cF_0$ lies in $\cF_M$. The continuous inclusions and the norm formula are the pullbacks of the corresponding inclusions and formula for $\cF_\infty\subset\cF_M\subset\cF_0$.
\end{proof}

\subsection{The side of kernel functions}

We turn now to the space $\mathcal A_0$ and its corresponding subspaces. Recall that the map
\[
\mathcal A_0\longrightarrow \cS_0,
\qquad
\psi\longmapsto S_\psi
\]
is bijective with inverse given by
\[
S\longmapsto \psi_S,\qquad \psi_S(x,y,v)=(S K_{0,v})(x,y).
\]
Moreover,
\[
\|\psi(\cdot,\cdot,y)\|_H=\|S_\psi K_{0,y}\|_H.
\]
Hence, the natural topology in $\cA_0$ is the initial topology $\tau_{\cA_0}$ generated by the seminorms
\[
\psi\longmapsto \|\psi(\cdot,\cdot,y)\|_H,\qquad y\in Y.
\]
For notational clarity on the kernel side, we write $a_\psi$ for the symbol $b_\psi$ associated with $\psi$.

As expected from the bounded operator case, the formal product $\odot$ from \eqref{eq:def odot product}, when it belongs to $\cA_0$, corresponds to pointwise multiplication on the symbol side. More precisely, if $\varphi\odot\psi\in \cA_0$, then one can check that
\[
S_{\varphi\odot\psi}=R^*M_{a_{\varphi}a_\psi}R
=S_\varphi\star S_\psi,
\]
on the common domain $\cD_0$.

Consider now the space $\cA_M$ of all functions $\psi\in \cA_0$ such that $a_\psi\in\cF_M$. Equivalently, this is the space of all $\psi\in\cA_0$ such that $S_\psi\in \cS_M$.
For $\varphi,\psi\in\cA_M$ we define $\varphi\odot\psi$ to be the unique element of $\cA_0$ whose symbol is $a_\varphi a_\psi$. This transported product agrees with the formal product \eqref{eq:def odot product} whenever the latter defines an element of $\cA_0$.
Likewise, for $\psi\in\cA_M$ and $m\in\bN$, we define $\psi^{\odot m}$ as the unique kernel in $\cA_0$ whose symbol is $a_\psi^m$.
By Proposition \ref{prop:characterization FM}, the condition $\psi\in\cA_M$ is equivalent to
\[
\|\psi^{\odot m}(\cdot,\cdot,y)\|_H<\infty,
\]
for all $y\in Y$ and $m\in\bN$.

In analogy to the previous cases, we endow $\cA_M$ with the initial topology $\tau_{\cA_M}$ generated by all seminorms
\[
\psi\longmapsto \|\psi^{\odot m}(\cdot,\cdot,y)\|^{1/m}_H,\qquad y\in Y,\ m\in\bN.
\]
These are seminorms because the symbol correspondence gives
\[
\|\psi^{\odot m}(\cdot,\cdot,y)\|_H
=
\|S_\psi^{\star m}K_{0,y}\|_H
=
\|M_{a_\psi^m}q_\bullet(y)\|_{L^2(\Omega)}
=
\|a_\psi\|_{y,m}^m.
\]
The involution on $\cA_M$ is defined by \eqref{eq:def involution A0} and coincides with the one transported from $\cF_M$ and $\cS_M$:
\begin{equation}
a_{\psi^\dagger}=\overline{a_\psi},\qquad S_{\psi^\dagger}=S_\psi^\dagger.
\end{equation}

\begin{thm}
\label{thm:statements A0}
The following statements hold.
\begin{enumerate}[label=(\arabic*)]
    \item The space $(\cA_0,\tau_{\cA_0})$ is a complete topological vector space that is isomorphic to $(\cF_0,\tau_{\cF_0})$ via the map
    \[
    \psi\mapsto a_\psi.
    \]

    \item The space $(\cA_M,\tau_{\cA_M})$ is a commutative complete topological $*$-algebra with respect to the product $\odot$ and the involution given by \eqref{eq:def involution A0}.
    
    \item The isomorphism $\psi\mapsto a_\psi$ restricts to an isomorphism of topological $*$-algebras making $\cA_M$ isomorphic to $\cF_M$.
    
    \item The algebra $\cA_M$ is largest in the following precise sense: under the symbol map, it corresponds to the largest pointwise multiplicative subalgebra $\cF_M$ of $\cF_0$, and the product $\odot$ is transported from pointwise multiplication of symbols.
    
    \item The following inclusions hold and are continuous:
    \[
    \cA_\infty\subset \cA_M\subset \cA_0.
    \]
    
    \item For every $\psi\in \cA_\infty$ one has
    \[
    \|\psi\|_{\cA_\infty}=\sup_{y\in Y_0}\lim_{m\to\infty}\|\psi^{\odot m}(\cdot,\cdot,y)\|_H^{1/m}.
    \]
\end{enumerate}
\end{thm}
\begin{proof}
The map $\Lambda:\cA_0\to\cF_0$, $\Lambda(\psi)=a_\psi$, is bijective by Corollary \ref{cor:equiv S0 A0 F0} and the bijection $\psi\mapsto S_\psi$ between kernels and operators. Furthermore,
\[
\|\psi(\cdot,\cdot,y)\|_H=\|S_\psi K_{0,y}\|_H=\|a_\psi\|_y,
\]
so $\tau_{\cA_0}$ is the pullback of $\tau_{\cF_0}$. This proves completeness of $\cA_0$.

By definition, $\cA_M=\Lambda^{-1}(\cF_M)$, and the preceding seminorm identity shows that $\tau_{\cA_M}$ is the pullback of $\tau_{\cF_M}$. Moreover,
\[
\Lambda(\varphi\odot\psi)=\Lambda(\varphi)\Lambda(\psi),
\qquad
\Lambda(\psi^\dagger)=\overline{\Lambda(\psi)}.
\]
Thus $\Lambda$ restricts to an isomorphism of topological $*$-algebras from $\cA_M$ onto $\cF_M$. The algebraic, topological and completeness
follow from the corresponding assertions for $\cF_M$.

The maximality statement is again the pullback of Proposition \ref{prop:characterization FM}: any subspace of $\cA_0$ closed under the transported product has a symbol set contained in $\cF_M$. The continuous inclusions and the formula for the norm on $\cA_\infty$ follow from the corresponding symbol-side statements.
\end{proof}
We remark that Bais, Maximenko and Venku Naidu \cite{BaisMaximenkoVenkuNaidu2025} endowed the algebra $\cA_\infty$ (denoted by $\cA$ in that work) with the norm given by transporting the norm from $L^\infty(\Omega)$. In contrast, the last formula in the above theorem provides a direct and intrinsic way to compute this norm.

Summarizing the results from the previous sections, we present the following diagram. The horizontal arrows are continuous inclusions while the vertical ones are isomorphisms:
\[
\begin{tikzcd}
L^\infty(\Omega)=\mathcal{F}_\infty \arrow[r, hook] \arrow[d] &
\mathcal{F}_M \arrow[r, hook] \arrow[d] &
\mathcal{F}_0 \arrow[d] \\
\cC(\rho)=\mathcal{S}_\infty \arrow[r, hook] \arrow[d] &
\mathcal{S}_M \arrow[r, hook] \arrow[d] &
\mathcal{S}_0 \arrow[d] \\
\mathcal{A}_\infty \arrow[r, hook] &
\mathcal{A}_M \arrow[r, hook] &
\mathcal{A}_0
\end{tikzcd}
\]

Note that our results in this case give an answer to the problem posed in \cite[Remark 7.12]{BaisMaximenkoVenkuNaidu2025} asking whether it was possible to extend the multiplication in $\mathcal A$ to an operation on the space $\cA_0$, as the following result shows.
\begin{cor}
\label{cor:extension odot}
    The following conditions are equivalent:
    \begin{enumerate}[label=(\arabic*)]
        \item The formal product \eqref{eq:def odot product} is $\cA_0$-valued for every pair $\varphi,\psi\in\cA_0$ and agrees with the product transported from pointwise multiplication of symbols; equivalently, the bounded kernel product on $\cA_\infty$ extends to an operation
        \begin{equation*}
        \odot\colon \cA_0\times\cA_0\to \cA_0
        \end{equation*}
        compatible with the symbol correspondence.
        \item $\cA_M=\cA_0$.
        \item $\cS_M=\cS_0$.
        \item $\cF_M=\cF_0$.
    \end{enumerate}
\end{cor}
\begin{proof}
If (1) holds, then the symbol of $\varphi\odot\psi$ is $a_\varphi a_\psi$ for every $\varphi,\psi\in\cA_0$. Hence the symbol set $\cF_0$ is closed under pointwise multiplication. By Proposition \ref{prop:characterization FM}, every multiplicative subalgebra of $\cF_0$ is contained in $\cF_M$, while always $\cF_M\subset\cF_0$. Thus $\cF_M=\cF_0$. Conversely, if $\cF_M=\cF_0$, then Theorem \ref{thm:statements A0} gives $\cA_M=\cA_0$, and the transported product defines the desired operation on all of $\cA_0$; by the discussion preceding Theorem \ref{thm:statements A0}, it agrees with the formal integral product whenever the latter is interpreted as an element of $\cA_0$.

The equivalences with $\cA_M=\cA_0$ and $\cS_M=\cS_0$ follow from the topological $*$-algebra isomorphisms between the symbol, operator and kernel models in Theorems \ref{thm:statements S0} and \ref{thm:statements A0}.
\end{proof}

\begin{cor}
    If $(\cA_0,\tau_{\cA_0})$ is a Fréchet space and $G$ is not compact, then the bounded kernel product cannot be extended to a symbol-compatible operation $\odot\colon \cA_0\times\cA_0\to\cA_0$.
\end{cor}
\begin{proof}
By Theorem \ref{thm:statements A0}, the space $(\cA_0,\tau_{\cA_0})$ is topologically isomorphic to $(\cF_0,\tau_{\cF_0})$. Hence $\cF_0$ is Fr\'echet. If such a symbol-compatible extension existed, Corollary \ref{cor:extension odot} would imply $\cF_M=\cF_0$, and then Theorem \ref{thm:particular characterization A0=AM} would force $G$ to be compact. This contradicts the hypothesis.
\end{proof}

\section{Three examples}
\label{sec:examples}

In this section, we illustrate how the results obtained above are manifested in three specific examples from \cite{BaisMaximenkoVenkuNaidu2025}. We follow closely the definitions and notation introduced therein. To keep the paper at a reasonable length, we restrict ourselves to these representative examples and do not discuss the underlying constructions in detail. For further information, we refer the reader to \cite{BaisMaximenkoVenkuNaidu2025,HerreraYanezMaximenkoRamosVazquez2022}. A more detailed analysis of these and other examples is left for future work.

\subsection{Vertical operators on the Bergman space over the upper half-plane}

Let $\Pi=\{z\in\bC\colon \operatorname{Im}(z)>0\}$ denote the upper half-plane and let $L^2_{\text{hol}}(\Pi)$ be the space of holomorphic functions on $\Pi$ which are square integrable with respect to the usual area Lebesgue measure restricted to $\Pi$. As is well-known, $L^2_{\text{hol}}(\Pi)$ is a reproducing kernel Hilbert space. Identifying $\Pi=\bR\times\bR_+$, this space is an example of the RKHS considered in this paper.

Here $G=\bR$, $Y=\bR_+$, $\nu$ is the usual Lebesgue measure and $\lambda$ is the Lebesgue measure restricted to $\bR_+$. The group $\bR$ acts on $f\in L^2_{\text{hol}}(\Pi)$ by
\[
(\rho(t)f)(z)=f(z-t),\qquad t\in \bR,\ z\in\Pi.
\]

Its reproducing kernel is given by
\[
K_{u,v}(x,y)=-\frac{1}{\pi((x-u)+i(y+v))^2},\qquad (u,v),(x,y)\in \bR\times\bR_+.
\]
In \cite{HerreraYanezMaximenkoRamosVazquez2022} it was shown that the Fourier transform of $K_{(0,v)}(\cdot,y)$ is
\[
L_{\xi,v}(y)=4\pi\xi e^{-2\pi(y+v)\xi}\chi_{\bR_+}.
\]
Therefore, $\Omega=\bR_+$ and
\[
q_\xi(y)=2\sqrt{\pi\xi}e^{-2\pi y\xi}.
\]
Note that in this case the common domain for the operators in the spaces $\cS_M$ and $\cS_0$ is
\[
R(\cD_0)=
\operatorname{span}
\left\{
\xi\mapsto 2\sqrt{\pi\xi} e^{i2\pi x\xi}e^{-2\pi y\xi}\colon x\in \bR,\ y\in Y
\right\}.
\]
We obviously have $q_\bullet(y_2)<q_\bullet(y_1)$ for $y_2> y_1$, which implies that the sequence of seminorms $(\|\cdot\|_{1/k})_{k\in \bN}$ generates the topology $\tau_{\cF_0}$. Hence, $(\cF_0,\tau_{\cF_0})$ is a Fréchet space and $(\cF_M,\tau_{\cF_M})$ a Fréchet $*$-algebra.
By Corollary \ref{cor:G not compact both inclusions strict}, we have
\[
L^\infty(\Omega)\subsetneq \cF_M\subsetneq \cF_0.
\]
However, we can show this directly.

In \cite[Remark 8.2]{BaisMaximenkoVenkuNaidu2025} the authors showed that $\cA\subsetneq\cA_0$ (equivalently, that $\cF=L^\infty(\Omega)\subsetneq \cF_0$) by showing that the unbounded function $c(\xi)=\xi$ belongs to $\cF_0$. Note that $c^m\in L^2(\Omega,q_\bullet(y)^2d\widehat{\nu})$ for all $m\in\bN$. Therefore, indeed one has $c\in \cF_M$.

An explicit example showing that $\cF_M\subsetneq \cF_0$ is provided by
\[
d(\xi)=\xi^{-3/4}\chi_{[0,1]}(\xi).
\]
For every $y>0$,
\[
\|d\|_y^2
=4\pi\int_0^1 \xi^{-1/2}e^{-4\pi y\xi}\,d\xi
<\infty,
\]
since $e^{-4\pi y\xi}\le 1$ on $[0,1]$. On the other hand,
\[
\|d^2\|_y^2
=4\pi\int_0^1 \xi^{-2}e^{-4\pi y\xi}\,d\xi
\ge
4\pi e^{-4\pi y}\int_0^1 \xi^{-2}\,d\xi
=\infty.
\]
Hence
\[
d\in \cF_0\setminus \cF_M.
\]

\subsection{Radial operators on the Bergman space over the unit disk}

Consider now the unit disk $\bD=\{z\in\bC\colon |z|<1\}$ and the Bergman space $L^2_{\text{hol}}(\bD)$ consisting of holomorphic functions on $\bD$ which are square integrable with respect to the usual area Lebesgue measure restricted to $\bD$. This is also a RKHS with reproducing kernel
\[
K_{w}(z)=\frac{1}{\pi(1-z\overline{w})^2},\qquad z,w\in\bD.
\]

Passing to polar coordinates, we can identify $\bD\cong [0,1)\times\bT$ and regard $L^2_{\text{hol}}(\bD)$ as a subspace of
$L^2([0,1)\times\bT,\nu\otimes\lambda)$, where $d\nu$ is the Haar measure of $\bT$ and $d\lambda(r)=rdr$.
Hence, in this case $G=\bT$ and $Y=[0,1)$. The action of $\bT$ on $f\in L^2_{\text{hol}}(\bD)$ is given by
\[
(\rho(t)f)(z)=f(t^{-1}z),\qquad t\in\bT,\ z\in \bD.
\]
As was shown in \cite{HerreraYanezMaximenkoRamosVazquez2022,BaisMaximenkoVenkuNaidu2025}, for this space we have $\Omega=\bN_0\coloneqq\bN\cup\{0\}$,
\[
L_{k,s}(r)=2(k+1)(rs)^k\chi_{\bN_0}(k)
\]
and
\[
q_k(r)=\sqrt{2(k+1)}r^k.
\]
In this case, we call the elements of $\cS_0$ \emph{radial operators}. The space $\cF_0$ in this case consists of sequences $(a_k)_{k\in \bN_0}$ such that
\begin{equation}
  \label{eq:Bergman radial def |ak|y}
  \|(a_k)\|_y^2=\sum_{k=0}^\infty2(k+1)|a_k|^2y^{2k}<\infty,\qquad\forall y\in[0,1).
\end{equation}

Note that, since $\bT$ is compact, Theorem \ref{thm:particular characterization A0=AM} and Corollary \ref{cor:A=AM implies compactness} do not apply. However, one can show explicitly that $\cF_\infty\subsetneq\cF_M$ but $\cF_M=\cF_0$.

The first inclusion follows easily from Proposition \ref{prop:F=F0 iff F=FM}, since it suffices to prove that $\cF_\infty\subsetneq\cF_0$. Consider, for example, the unbounded sequence $(a_k)_{k\in\bN_0}$ given by $a_k=\sqrt{k}$, for all $k\in\bN_0$. Then for any $y\in[0,1)$ one has
\[
\|(a_k)\|_y^2=
\sum_{k=0}^\infty
2k(k+1)y^{2k}=\frac{4y^2}{(1-y^2)^3}<\infty.
\]

As for the other inclusion, let $(a_k)\in \cF_0$. Then \eqref{eq:Bergman radial def |ak|y} holds for all $y\in [0,1)$ and this implies that the radius of convergence of the power series with coefficients $2(k+1)|a_k|^2$ is at least $1$.
By the Cauchy--Hadamard theorem, one has
\[
\limsup_{k\to\infty}\sqrt[k]{2(k+1)|a_k|^2}
=
\limsup_{k\to\infty}\sqrt[k]{|a_k|^2}
\leq 1.
\]
Hence, for every $m\in\bN$ one has
\[
\limsup_{k\to\infty}\sqrt[k]{2(k+1)|a_k|^{2m}}
=
\limsup_{k\to\infty}\sqrt[k]{|a_k|^{2m}}
\leq 1,
\]
which implies that, for every $y\in [0,1)$,
\[
\|(a_k)^m\|_y^2=\sum_{k=0}^\infty 2(k+1)|a_k|^{2m}y^{2k}<\infty.
\]
Thus $(a_k)\in\cF_M$, and therefore $\cF_M=\cF_0$.

\subsection{Radial operators on the Fock space}

Finally, let $F^2(\bC)$ denote the Fock space consisting of all holomorphic functions $f$ in $\bC$ such that
\[
\|f\|_{F^2(\bC)}^2\coloneqq
\frac{1}{\pi}\int_{\bC}|f(z)|^2e^{-|z|^2}dA(z),
\]
where $dA$ is the usual Lebesgue area measure on $\bC$.
This is a RKHS with reproducing kernel
\[
K_w(z)=e^{z\overline{w}},\qquad z,w\in \bC.
\]
One can easily see that the Fock space is invariant under the action of $\bT$ given by
\[
\rho(t)f(z)=f(t^{-1}z),\qquad t\in \bT,\ z\in \bC.
\]
As was done for the Bergman space on $\bD$, passing to polar coordinates, one can see that $F^2(\bC)$ can be embedded in the space
\[
L^2\left(\bT\times\bR_+,d\nu\otimes\frac{1}{\pi}re^{-r^2}dr\right),
\]
where $\nu$ is again the Haar measure of $\bT$.

Expanding $K_w(z)=e^{z\overline{w}}=\sum^\infty_{k=0}\frac{(z\overline{w})^k}{k!}$, one can see that $\Omega=\bN_0$ with
\[
L_{k,s}(r)=2\frac{(rs)^k}{k!}\chi_{\bN_0}(k),
\]
and thus
\[
q_k(r)=\sqrt{\frac{2}{k!}}r^k.
\]
Therefore, the space $\cF_0$ consists of all sequences $(a_k)_{k\in\bN_0}$ such that
\[
\|(a_k)\|_y^2=\sum_{k=0}^\infty\frac{2|a_k|^2y^{2k}}{k!}<\infty,
\qquad y\in\bR_+.
\]
By a similar argument as in the Bergman space on $\bD$, one can show that $\cF_\infty\subsetneq\cF_M$; for instance, $a_k=k$ gives an unbounded sequence belonging to $\cF_M$.

However, in this case we also have $\cF_M\subsetneq\cF_0$. Consider the sequence $(a_k)_{k\in\bN_0}$ given by
\[
a_k=(k!)^{1/4},\qquad k\in\bN_0.
\]
Then for every $y>0$ one has
\begin{equation}
    \label{eq:norm Fock ex FM not F0}
    \|(a_k)\|_y^2=
\sum_{k=0}^\infty\frac{2y^{2k}}{\sqrt{k!}}.
\end{equation}
By Stirling's approximation,
\[
\limsup_{k\to\infty}
\sqrt[k]{\frac{2}{\sqrt{k!}}}
=0.
\]
Hence the Cauchy--Hadamard theorem implies that the series in \eqref{eq:norm Fock ex FM not F0} converges for every $y\geq 0$, whence $(a_k)\in \cF_0$.

However, one has
\[
\|(a_k)^2\|_y^2=
\sum_{k=0}^\infty 2y^{2k},
\]
which diverges for $y\geq 1$. Thus $(a_k)\in\cF_0\setminus\cF_M$.

Comparing this and the previous example shows that compactness of the group $G$ is not enough to guarantee that $\cF_M=\cF_0$.

\section{Open questions}
\label{sec:questions}

We end this work by stating the open questions and problems collected throughout this paper.

\begin{enumerate}
    \item Do the hypotheses on the space $L^2(G\times Y,\nu\times\lambda)$ suffice to prove that $\cF_0$ is a Fréchet space?

    \item Find a complete characterization of the equalities $\cF_\infty=\cF_M$ and $\cF_M=\cF_0$. Does one need the Fréchet assumption?

    \item Give an explicit characterization of the space $\cD=\bigcap_{a\in\cF_M}\operatorname{Dom}(M_a^{\text{max}})$ from Remark \ref{rem:S0 is Ostar algebra}.

    \item If the direct integral decomposition of $H$ is not one-dimensional, one may still be able to define the space $\cF_0$ through diagonalization of translation-invariant operators. These spaces will consist of matrix-valued functions (see \cite{HerreraYanezMaximenkoRamosVazquez2022}). How far can one extend the results from this paper to this non-commutative setting?

    \item Characterize when special classes of operators (e.g. Toeplitz operators) belong to the class $\cS_M$.
\end{enumerate}

\section*{Acknowledgements}

This work was supported by SECIHTI through the program
\emph{Estancias posdoctorales por México}, CVU 860740.

\printbibliography

\section*{Author information}

\noindent\textbf{Miguel Angel Rodriguez Rodriguez}\\
Escuela Superior de Física y Matemáticas, Instituto Politécnico Nacional\\
Ciudad de México, C.P. 07738, Mexico\\
\href{mailto:miarodriguezro@ipn.mx}{miarodriguezro@ipn.mx}\\
\href{https://orcid.org/0000-0002-5124-8271}{ORCID: 0000-0002-5124-8271}

\end{document}